\documentclass[11pt,a4paper]{article}
\usepackage{graphicx}
\usepackage{latexsym}
\usepackage[all]{xy}
\usepackage{amsfonts}
\usepackage{amsthm}
\usepackage{amsmath}
\usepackage{amssymb}

\def\<{{\langle}}
\def\>{{\rangle}}

\def\note#1{{}}

\def\note#1{}

\def\rend#1#2{{{\rm End}\sb{#1}(#2)}}

\def\beq{\begin{equation}}
\def\eeq{\end{equation}}

\def\id{{I}}

\def\ot{{\otimes}}

\newcounter{zlist}

\def\Label#1{\label{#1}\ifmmode\llap{[#1] }\else 
\marginpar{\smash{\hbox{\tiny [#1]}}}\fi}
\def\Label{\label}

\newtheorem{proposition}{Proposition}[section]

\newtheorem{theorem}[proposition]{Theorem}

\theoremstyle{definition}
\newtheorem{definition}[proposition]{Definition}

\theoremstyle{remark}
\newtheorem{remark}[proposition]{Remark}

\newcounter{c}

\newcommand{\etyk}[1]{\vspace{-7.4mm}$$\begin{equation}\Label{#1}
\addtocounter{c}{1}}
\renewcommand{\]}{\ifnum \value{c}=1 $$\else \end{equation}\fi}
\begin{document}

\title{YANG--BAXTER OPERATORS FROM $( \mathbb{G},\theta )$-LIE ALGEBRAS}
\author{F. F. Nichita\,$^1$
and Bogdan Popovici\,$^2$\\
$^1$ Institute of Mathematics "Simion Stoilow" of the Romanian Academy \\ 
P.O. Box 1-764, RO-014700 Bucharest, Romania \\
E-mail: Florin.Nichita@imar.ro\\
$^2$ Horia Hulubei National Institute for Physics and Nuclear Engineering,\\
P.O.Box MG-6, Bucharest-Magurele, Romania\\
E-mail: popobog@theory.nipne.ro
}

\date{October 2010}
\maketitle
\begin{abstract}

The $(\mathbb{G},\theta)$-Lie algebras are structures which unify 
the Lie algebras and Lie superalgebras. We 
use them to produce solutions for the  
quantum Yang--Baxter equation. 
The constant and the  spectral-parameter 
Yang-Baxter equations and
Yang-Baxter systems are also studied.

\end{abstract}

\section{Introduction }

The $(\mathbb{G},\theta)$-Lie algebras are structures which unify 
the Lie algebras and Lie superalgebras, and, in this paper,
they will be used to produce solutions for the celebrated 
quantum Yang--Baxter equation (QYBE). 
The 
theory of integrable Hamiltonian
systems makes great use of the solutions of the one-parameter form of the
QYBE (which are related to the two-parameter form of the QYBE), 
since coefficients of the power series expansion of
such a solution give rise to commuting integrals of motion.
{\em Yang--Baxter systems} emerged from the study of quantum
integrable systems, as generalizations of the QYBE related to 
nonultralocal models.

This paper presents some of the latest results
on Yang-Baxter operators from algebra structures
and related topics, such as
enhanced
versions of Theorem 1 (from
\cite{tbh}). Also, we study
Yang-Baxter operators from  Lie superalgebras and from
$(\mathbb{G},\theta)$-Lie algebras. 
The following authors constructed 
Yang-Baxter
operators from Lie (co)algebras and Lie superalgebras before:
 \cite{mj}, \cite{ba}, \cite{nichita}, etc.
We extend some of the above results to Yang-Baxter
systems and spectral-parameter dependent Yang-Baxter equations.

\section{The constant QYBE}

Throughout this paper $ k $ is a field. All tensor products appearing in this paper are defined over $k$.
For $ V $ a $ k$-space, we denote by
$ \   \tau : V \otimes V \rightarrow V \ot V \  $ the twist map defined by $ \tau (v \ot w) = w \ot v $, and by $ I: V \rightarrow V $
the identity map of the space V.

We use the following notations concerning the Yang-Baxter equation.

If $ \  R: V \ot V \rightarrow V \ot V  $
is a $ k$-linear map, then
$ {R^{12}}= R \ot I , {R^{23}}= I \ot R ,
{R^{13}}=(I\ot\tau )(R\ot I)(I\ot \tau ) $.

\begin{definition}
An invertible  $ k$-linear map  $ R : V \ot V \rightarrow V \ot V $
is called a Yang-Baxter
operator if it satisfies the  equation
\begin{equation}  \label{ybeq}
R^{12}  \circ  R^{23}  \circ  R^{12} = R^{23}  \circ  R^{12}  \circ  R^{23}
\end{equation}
\end{definition}

\begin{remark}
The equation (\ref{ybeq}) is usually called the braid equation. It is a
well-known fact that the operator $R$ satisfies (\ref{ybeq}) if and only if
$R\circ \tau  $ satisfies
  the constant QYBE 
(if and only if
$ \tau \circ R $ satisfies
the constant QYBE):
\begin{equation}   \label{ybeq2}
R^{12}  \circ  R^{13}  \circ  R^{23} = R^{23}  \circ  R^{13}  \circ  R^{12}
\end{equation}

\end{remark}

\begin{remark}
(i) $ \   \tau : V \otimes V \rightarrow V \ot V \  $ is an example of a 
Yang-Baxter operator.

(ii) An exhaustive list of invertible solutions for (\ref{ybeq2}) in dimension 
2 is given in \cite{hi} and in the appendix of \cite{HlaSno:sol}.

(iii) Finding all Yang-Baxter operators in dimension greater than 2 is an 
unsolved problem.

\end{remark}

\bigskip

Let $A$ be a (unitary) associative $k$-algebra, and $ \alpha, \beta, \gamma \in k$. 
We define the
$k$-linear map:
$ \  R^{A}_{\alpha, \beta, \gamma}: A \ot A \rightarrow A \ot A, \ \ 
R^{A}_{\alpha, \beta, \gamma}( a \ot b) = \alpha ab \ot 1 + \beta 1 \ot ab -
\gamma a \ot b $.

\begin{theorem} (S. D\u asc\u alescu and F. F. Nichita, \cite{DasNic:yan})
\label{primat}
Let $A$ be an associative 
$k$-algebra with $ \dim A \ge 2$, and $ \alpha, \beta, \gamma \in k$. Then $ R^{A}_{\alpha, \beta, \gamma}$ is a Yang-Baxter operator if and only if one
of the following holds:

(i) $ \alpha = \gamma \ne 0, \ \ \beta \ne 0 $;

(ii) $ \beta = \gamma \ne 0, \ \ \alpha \ne 0 $;

(iii) $ \alpha = \beta = 0, \ \ \gamma \ne 0 $.

If so, we have $ ( R^{A}_{\alpha, \beta, \gamma})^{-1} = 
R^{A}_{\frac{1}{ \beta}, \frac{1}{\alpha}, \frac{1}{\gamma}} $ in cases (i) and
(ii), and $ ( R^{A}_{0, 0, \gamma})^{-1} = 
R^{A}_{0, 0, \frac{1}{\gamma}} $ in case (iii).
\end{theorem}

\begin{remark}

The Yang--Baxter equation plays an important role in knot theory. 
Turaev has described a general scheme to derive an invariant of 
oriented links from a Yang--Baxter operator, provided this 
one can be ``enhanced''.
In \cite{mn}, we considered the problem of applying Turaev's method to the
Yang--Baxter operators derived from algebra structures presented in
the above theorem. 
We concluded that 
Turaev's procedure invariably produces from any of those enhancements the 
Alexander polynomial of knots. 

\end{remark}

\bigskip

\begin{remark}
Let us observe that $ R'=  R^{A}_{\alpha, \beta, \alpha} \circ \tau $ is a solution for the equation
(\ref{ybeq2}). In dimension two, after
getting rid of the auxiliary parameters, we obtain
the simplest form of $ R' $:
\begin{equation} \label{rmatcon2}
\begin{pmatrix}
1 & 0 & 0 & 0\\
0 & 1 & 0 & 0\\
0 & 1-q  & q & 0\\
\eta & 0 & 0 & -q
\end{pmatrix}
\end{equation}
where $ \eta \in \{ 0, \ 1 \} $, and $q \in k - \{ 0 \}$.
The matrix form (\ref{rmatcon2}) was obtained as a consequence of the fact
that isomorphic algebras produce isomorphic Yang-Baxter operators, and it is
compatible with Remark 2.3 (ii).
\end{remark}

\section{The two-parameter form of the QYBE}

Formally, a colored Yang-Baxter operator is defined as a function $ R
:X\times X \to \rend k {V\otimes V}, $ where $X$ is a set and $V$ is a
finite dimensional vector space over a field $k$. 
Thus, for any $u,v\in X$,
$R(u,v) : V\otimes V\to V\otimes V$ is a linear operator. 
We consider three operators acting on a triple
tensor product $V\otimes V\otimes V$, $R^{12}(u,v) = R(u,v)\otimes
\id$, $R^{23}(v,w)= \id\otimes R(v,w)$, and similarly $R^{13}(u,w)$ as
an operator that acts non-trivially on the first and third factor in
$V\otimes V\otimes V$. 

$R$
satisfies the two-parameter form of the QYBE if:
\begin{equation}\label{yb} 
R^{12}(u,v)R^{13}(u,w)R^{23}(v,w) = R^{23}(v,w)
R^{13}(u,w)R^{12}(u,v)
\end{equation} 
$ \forall \ u,v,w\in X$.

\bigskip

\begin{theorem} (F. F. Nichita and D. Parashar, \cite{np})
Let $A$ be an associative 
$k$-algebra with $ \dim A \ge 2$, and
$ X \subset k $. Then,
for any two parameters $p,q\in k$, the function
$R:X\times X\to \rend k {A\otimes A}$ defined by
\begin{equation}\label{rsol} 
R(u,v)(a\otimes b) =p(u-v)1\otimes ab + q(u-v)ab\otimes 1 -(pu-qv)b\otimes a,
\end{equation}
satisfies the colored QYBE (\ref{yb}).
\end{theorem}

\begin{remark}
If $ \ pu \neq qv $ and $ \ qu \neq pv $ then the operator
($\ref{rsol}$) is invertible. Moreover, the following formula holds:
$ \ \ \ R^{-1}(u,v)(a\otimes b) = \frac{p(u-v)}{(qu-pv)(pu-qv)}ba\otimes 1 + 
\frac{q(u-v)}{(qu-pv)(pu-qv)}1\otimes ba - \frac{1}{(pu-qv)}b\otimes a $.

\end{remark}

\bigskip

Algebraic manipulations of the previous theorem lead to the following result.

\begin{theorem} \label{top}
Let $A$ be an associative 
$k$-algebra with $ \dim A \ge 2$ and
 $q\in k$. Then the operator
\begin{equation}\label{slsol} 
S( \lambda )(a\otimes b) = (e^{\lambda} - 1)1\otimes ab 
+ q(e^{\lambda} - 1)ab\otimes 1 -(e^{\lambda}-q)b\otimes a
\end{equation}
satisfies the one-parameter form of the Yang-Baxter equation:
$$S^{12} (\lambda_{1} - \lambda_{2}) S^{13} (\lambda_{1} - \lambda_{3}) S^{23}(\lambda_{2} - \lambda_{3})=$$
\begin{equation}\label{onepara}
= S^{23} (\lambda_{2} - \lambda_{3}) S^{13} (\lambda_{1} - \lambda_{2})
S^{12} (\lambda_{1} - \lambda_{2}).
\end{equation}
If $ \ e^{\lambda} \neq q  , \ \frac{1}{q} $, $ \ $ then the operator
($\ref{slsol}$) is invertible.
 
Moreover, the following formula holds:

$ \ \ \ S^{-1}(\lambda)(a\otimes b) = \frac{e^{\lambda}-1}{(qe^{\lambda}-1)(e^{\lambda}-q)}ba\otimes 1 + 
\frac{q(e^{\lambda}-1)}{(qe^{\lambda}-1)(e^{\lambda}-q)}1\otimes ba 
- \frac{1}{e^{\lambda}-q}b\otimes a $.
\end{theorem}

\begin{remark}

The operator from Theorem \ref{top} can be obtained from
Theorem \ref{primat} and the {\bf Baxterization} procedure from \cite{defk} (page 22). 

Hint: Consider the operator
$ \  R^{A}_{q, \frac{1}{q}, \frac{1}{q} }: A \ot A \rightarrow A \ot A, \ \ 
 a \ot b \mapsto q ab \ot 1 +  \frac{1}{q} \ot ab -
 \frac{1}{q} a \ot b $ and its inverse, $ R^{A}_{q, \frac{1}{q}, q } $.

\end{remark}

\section{Yang-Baxter systems}
From the physical point of view the above relations are used to study
a certain class of quantum integrable systems, the ultralocal models
\cite{Fad,Kul:Skl}. However, interesting physical models which  have nonultralocal 
interactions appear, and they require the study of extensions of the 
QYBE \cite{HlaKun:qua, HlaSno:sol}. In the following we describe 
the Yang-Baxter systems in terms of the Yang-Baxter commutators.

Let $V$, $V'$, $V''$ be finite dimensional
vector spaces over the field $k$, and let $R: V\ot
V' \rightarrow V\ot V'$, $S: V\ot V'' \rightarrow V\ot V''$ and $T:
V'\ot V'' \rightarrow V'\ot V''$ be three linear maps.
The {\em Yang--Baxter
commutator} is a map $[R,S,T]: V\ot V'\ot V'' \rightarrow V\ot V'\ot
V''$ defined by \beq [R,S,T]:= R^{12} S^{13} T^{23} - T^{23} S^{13}
R^{12}. \eeq 
Note that $[R,R,R] = 0$ is just a short-hand
notation for writing the constant QYBE (\ref{ybeq2}).

A system of linear
maps
$W: V\ot V\ \rightarrow V\ot V,\quad Z: V'\ot V'\ \rightarrow V'\ot
V',\quad X: V\ot V'\ \rightarrow V\ot V',$ is called a
$WXZ$--system if the
following conditions hold: \beq \label{ybsdoub} [W,W,W] = 0 \qquad
[Z,Z,Z] = 0 \qquad [W,X,X] = 0 \qquad [X,X,Z] = 0\eeq 

\begin{remark}
It 
was observed that $WXZ$--systems with invertible $W,X$ and $Z$ can
be used to construct dually paired bialgebras of the FRT type
leading to quantum doubles. The above is one type of a constant
Yang--Baxter system that has recently been studied in \cite{np} and
also shown to be closely related to entwining structures \cite{bn}.
\end{remark}

\bigskip

\begin{theorem} (F. F. Nichita and D. Parashar, \cite{np})
Let $A$ be a $k$-algebra, and $ \lambda, \mu \in k$. The following is a 
$WXZ$--system:

$ W : A \ot A \rightarrow A \ot A, \ \ 
W(a \ot b)= ab \ot 1 + \lambda 1 \ot ab - b \ot a $,

$ Z : A \ot A \rightarrow A \ot A, \ \ 
Z(a \ot b)= \mu ab \ot 1 +  1 \ot ab - b \ot a $,

$ X : A \ot A \rightarrow A \ot A, \ \ 
X(a \ot b)= ab \ot 1 +  1 \ot ab - b \ot a $.

\end{theorem}

\begin{remark}
Let $R$ be a 
solution for the two-parameter form of the QYBE, i.e.
$ \ R^{12}(u,v)R^{13}(u,w)R^{23}(v,w) = R^{23}(v,w)
R^{13}(u,w)R^{12}(u,v)
 \ \  \forall \ u,v,w\in X$.

Then, if we fix $ s, t \in X$, we obtain the following
$WXZ$--system:

$W= R(s, s) $,
$ \ X= R(s, t) $ and
$\ Z= R(t, t) $.

\end{remark}

\section{Lie superalgebras}

Using some of the above techniques we now present enhanced
versions of Theorem 1 (from
\cite{tbh}).\\

\begin{theorem}  (F. F. Nichita and B. P. Popovici, \cite{nipo})
Let $V = W \oplus kc $ be a $k$-space, 
and $ f, g : V \ot V \rightarrow V $ $k$-linear maps such that
$ f, g = 0 $ on $ V \ot c + c \ot V $.
Then,
$ R: V \ot V \rightarrow V \ot V, \ 
R(v \ot w)= f(v \ot w) \ot c + c \ot g(v \ot w) $
is a solution for QYBE (\ref{ybeq2}).

\end{theorem}

\bigskip
\begin{definition}
A Lie superalgebra is a (nonassociative) $\mathbb{Z}_2$-graded algebra, or superalgebra, 
over a field $k$ with the  Lie superbracket, satisfying the two conditions:
$$[x,y] = -(-1)^{|x||y|}[y,x]$$
$$ (-1)^{|z||x|}[x,[y,z]]+(-1)^{|x||y|}[y,[z,x]]+(-1)^{|y||z|}[z,[x,y]]=0 $$
where $x$, $y$ and $z$ are pure in the $\mathbb{Z}_2$-grading. Here, $|x|$ denotes the degree of $x$ (either 0 or 1). 
The degree of $[x,y]$ is the sum of degree of $x$ and $y$ modulo $2$.
\end{definition}

Let $ ( L , [,] )$ be a Lie superalgebra over $k$,
and  $ Z(L) = \{ z \in L : [z,x]=0 \ \ \forall \ x \in L \} $.

For $ z \in Z(L), \ \vert z \vert =0 $ and $ \alpha \in k $ we define:

$$ { \phi }^L_{ \alpha} \ : \ L \ot L \ \ \longrightarrow \ \  L \ot L $$

$$ 
x \ot y \mapsto \alpha [x,y] \ot z + (-1)^{ \vert x \vert \vert y \vert } y \ot x \ . $$

Its inverse is:

$$ {{ \phi }^L_{ \alpha}}^{-1} \ : \ L \ot L \ \ 
\longrightarrow \ \  L \ot L $$

$$x \otimes y \mapsto \alpha z \otimes [x, y] + (-1)^{ \vert x \vert \vert y \vert 
} y \otimes x$$

\begin{theorem} Let  $ ( L , [,] )$ be a Lie superalgebra 
and
$ z \in Z(L), \vert z \vert = 0  $, and $ \alpha \in k $. Then:
$ \ \ \ \  { \phi }^L_{ \alpha} $ is a YB operator.
\end{theorem}

\begin{proof}
The verification of the Yang-Baxter equation follows below: 
$$ { \phi }^{L_{23}}_{ \alpha} { \phi }^{L_{12}}_{ \alpha} { \phi }^{L_{23}}_{ \alpha} (a \ot b \ot c) = 
{ \phi }^{L_{12}}_{ \alpha} { \phi }^{L_{23}}_{ \alpha} { \phi }^{L_{12}}_{ \alpha}(a \ot b \ot c),  $$
\begin{eqnarray}
{ \phi }^{L_{23}}_{ \alpha} { \phi }^{L_{12}}_{ \alpha} { \phi }^{L_{23}}_{ \alpha} (a \ot b \ot c) =  
{ \phi }^{L_{23}}_{ \alpha} { \phi }^{L_{12}}_{ \alpha} (\alpha a \ot [b,c] \ot z + (-1)^{|b||c|} a \ot c \ot b) = \nonumber \\
{ \phi }^{L_{23}}_{ \alpha} (\alpha^2 [a,[b,c]]\ot z \ot z + (-1)^{|a||[b,c]|}\alpha [b,c] \ot a \ot z +  (-1)^{|b||c|}\alpha  [a,c] \ot z \ot b + \nonumber \\  
(-1)^{|b||c|} (-1)^{|a||c|} c \ot a \ot b  ) =  \alpha^2 [a,[b,c]]\ot z \ot z +  (-1)^{|a||[b,c]|} \alpha [b,c]\ot z \ot a + \nonumber \\  
(-1)^{|b||c|} \alpha [a,c]\ot b \ot z +  (-1)^{|b||c|} (-1)^{|a||c|}\alpha c\ot\ [a,b] \ot z + \nonumber \\   
(-1)^{|b||c|} (-1)^{|a||c|} (-1)^{|a||b|}  c \ot b \ot a \nonumber \\  
\end{eqnarray}

\begin{eqnarray}
{ \phi }^{L_{12}}_{ \alpha} { \phi }^{L_{23}}_{ \alpha} { \phi }^{L_{12}}_{ \alpha} (a \ot b \ot c) =  
{ \phi }^{L_{12}}_{ \alpha} { \phi }^{L_{23}}_{ \alpha} (\alpha [a,b] \ot z \ot c + (-1)^{|a||b|} b \ot a \ot c) = \nonumber \\
{ \phi }^{L_{12}}_{ \alpha} (\alpha [a,b]\ot c \ot z +  (-1)^{|a||b|}\alpha b \ot [a,c] \ot z +  (-1)^{|a||b|} (-1)^{|a||c|} b \ot c \ot a ) = \nonumber \\  
 \alpha^2 [[a,b],c]\ot z \ot z +  (-1)^{|[a,b]||c|} \alpha c\ot [a,b] \ot z + (-1)^{|a||b|} \alpha^2 [b,[a,c]]\ot z \ot z + \nonumber \\  
  (-1)^{|a||b|} (-1)^{|[a,c]||b|}\alpha [a,c]\ot\ b \ot z +  (-1)^{|a||b|} (-1)^{|a||c|} \alpha [b,c] \ot z \ot a  \nonumber \\  
+  (-1)^{|a||b|} (-1)^{|a||c|} (-1)^{|b||c|} c \ot b \ot a \nonumber \\  
\end{eqnarray}

i.e. 
\begin{eqnarray}
\alpha^2 [a,[b,c]]\ot z \ot z +  (-1)^{|a||[b,c]|} \alpha [b,c]\ot z \ot a +  (-1)^{|b||c|} \alpha [a,c]\ot b \ot z +  \nonumber  \\
(-1)^{|b||c|+|a||c|} \alpha c\ot\ [a,b] \ot z  =   \alpha^2 [[a,b],c]\ot z \ot z +  (-1)^{|[a,b]||c|} \alpha c\ot [a,b] \ot z + \nonumber \\ 
(-1)^{|a||b|} \alpha^2 [b,[a,c]]\ot z \ot z +   (-1)^{|a||b|} (-1)^{|[a,c]||b|}\alpha [a,c]\ot\ b \ot z + \nonumber \\
 (-1)^{|a||b|} (-1)^{|a||c|} [b,c] \ot z \ot a \nonumber \\
\end{eqnarray}

and the two terms are equal given the choice of $z$ and the Jacobi relations for the superalgebra:
\begin{eqnarray}
\alpha^2 [a,[b,c]]\ot z \ot z   =   \alpha^2 [[a,b],c]\ot z \ot z + (-1)^{|a||b|} \alpha^2 [b,[a,c]]\ot z \ot z \nonumber \\
\end{eqnarray}

It is easily checked that the two operators are inverse to each other.
\begin{eqnarray}
\phi_\alpha^L {\phi_\alpha^L}^{-1} (x \ot y) = \phi_\alpha^L (\alpha z \otimes [x, y] + (-1)^{ \vert x \vert \vert y \vert 
} y \otimes x) = \alpha^2 [z,[x,y]]\ot z + \nonumber \\ 
(-1)^{|[x,y]||z|} \alpha [x,y]\ot z +  (-1)^{|x||y|} \alpha [y,x]\ot z + (-1)^{|x||y|} (-1)^{|x||y|} x\ot y = x\ot y  \nonumber \\
\end{eqnarray}
\end{proof}

\begin{theorem}
 Let  $ ( L , [,] )$ be a Lie superalgebra, 
$ z \in Z(L), \vert z \vert = 0  $, 
$ X \subset k $,
and $ \alpha, \beta:X \times X \rightarrow k $. 
Then,
$R:X\times X\to \rend k {L \otimes L}$ defined by
\begin{equation}\label{Lsol} 
R(u,v)(a\otimes b) = \alpha(u,v)[a,b]\otimes z + \beta (u,v) (-1)^{|a||b|} a\otimes b,
\end{equation}
satisfies the colored QYBE (\ref{yb}) $ \iff 
 {\beta(u,w)} {\alpha(v,w)} =  {\alpha(u,w)} {\beta(v,w)} .$

\end{theorem}

\begin{proof} Following similar steps as in the previous proof, we
need that the next relations are true:
\begin{eqnarray}
\alpha(u,v)\alpha(v,w)\beta(u,w) [a,[b,c]]\ot z \ot z +  (-1)^{|b||c|} \alpha(u,v)\alpha(u,w)\beta(v,w) [[a,c],b]\ot z \ot z =  \nonumber\\  
\alpha(u,v)\alpha(u,w)\beta(v,w) [[a,b],c]\ot z \ot z \nonumber\\  
(-1)^{|a||[b,c]|} \alpha(v,w)\beta(u,w)\beta(u,v) a\ot [b,c] \ot z =(-1)^{|a||b|+|a||c|}\alpha(v,w)\beta(u,v)\beta(u,w) a\ot [b,c] \ot z \nonumber\\  
(-1)^{|b||c|+|[a,c]||b|} \alpha(u,w)\beta(u,v)\beta(v,w) [a,c]\ot b \ot z = (-1)^{|a||b|} \alpha(u,w)\beta(u,v)\beta(v,w) [a,c]\ot b \ot z\nonumber \\   
(-1)^{|b||c|+|a||c|} \alpha(u,v)\beta(u,w)\beta(v,w)   [a,b] \ot z \ot c = (-1)^{|[a,b]||c|} \alpha(u,v)\beta(u,w)\beta(v,w) [a,b]\ot z \ot c \nonumber \\   
\end{eqnarray}

It is easily observed that beside the first one all the relations are 
automatically fulfilled; as for the first relation, a sufficient condition is:
$$\frac{\alpha(v,w)}{\alpha(u,w)}=\frac{\beta(v,w)}{\beta(u,w)}.$$
For example, $ \alpha(u,v)= f(v)$ and $ \beta(u,v)= g(v)$ could be chosen.
\end{proof}

\begin{remark}
Letting $ u=v $ above, we obtain that:
$$ { \phi }^L_{ \alpha, \beta} \ : \ L \ot L \ \ \longrightarrow \ \  L \ot L $$

$$ 
x \ot y \mapsto \alpha [x,y] \ot z + (-1)^{ \vert x \vert \vert y \vert } \beta y \ot x \ . $$

and its inverse:

$$ {{ \phi }^L_{ \alpha, \beta}}^{-1} \ : \ L \ot L \ \ 
\longrightarrow \ \  L \ot L $$

$$x \otimes y \mapsto \frac{\alpha}{ {\beta}^2} z \otimes [x, y] + 
(-1)^{ \vert x \vert \vert y \vert 
} \frac{1}{ \beta } y \otimes x$$

are Yang-Baxter operators.
\end{remark}

\begin{remark}
Let us consider the above data and apply it to Remark 4.3.
Then, if we let $ s, t \in X$, we obtain the following
$WXZ$--system:

$W(a\otimes b) = R(s,s)(a\otimes b)=
 f(s)[a,b]\otimes z + g(s) (-1)^{|a||b|} a\otimes b, $ and

$Z(a\otimes b) = R(t,t)(a\otimes b)=   \ X(a\otimes b)= R(s,t)(a\otimes b)=
f(t)[a,b]\otimes z + g(t) (-1)^{|a||b|} a\otimes b$.

\end{remark}

\begin{remark}
The results presented in this section hold for Lie algebras as well.
This is a consequence of the fact that these operators restricted to the
first component of a Lie superalgebra have the same properties.

\end{remark}

\section{ $(\mathbb{G},\theta)$-Lie algebras}

We now consider the case of  $(\mathbb{G},\theta)$-Lie algebras as in \cite{Kanak}: a generalization of
Lie algebras and Lie superalgebras. 

A $(\mathbb{G},\theta)$-Lie algebras consists of a $\mathbb{G}$-graded 
vector space $L$, with $L=\oplus_{g\in\mathbb{G}}L_g$,  $\mathbb{G}$ a finite abelian group, a non associative 
multiplication $\langle ..,..\rangle : L \times L \to L$ respecting the graduation in the sense that
$\langle L_a,L_b\rangle \subseteq  L_{a+b}, \;\; \forall a,b\in \mathbb{G}$ and a function 
$\theta:\mathbb{G}\times\mathbb{G}\to C^{*} $ taking non-zero complex values. The following conditions
are imposed:
\begin{itemize}
  \item $\theta$-braided (G-graded) antisymmetry: $\langle x,y\rangle = -\theta(a,b)\langle y,x  \rangle$ 
  \item $\theta$-braided (G-graded) Jacobi id: $\theta(c,a)\langle x, \langle y,z\rangle\rangle + \theta(b,c)\langle z, \langle x,y\rangle\rangle +\theta(a,b)\langle y, \langle z,x\rangle\rangle =0 $
  \item $\theta : G \times G \to C^*$ color function 
$\left \{ \begin{array}{c}\theta(a+b,c) = 
\theta(a,c)\theta(b,c)\\ 
\theta(a,b+c) = \theta(a,b)\theta(a,c)\\   
\theta(a,b)\theta(b,a) = 1 \end{array}\right . $
\end{itemize}  
for all homogeneous $x\in L_a, y \in L_b, z\in L_c$ and $\forall a,b,c \in \mathbb{G}$.

\begin{theorem}
Under the above assumptions,

\begin{equation}\label{Lsol} 
R(x\otimes y) = \alpha[x,y]\otimes z + 
\theta(a,b) x\otimes y, 
\end{equation}
with $ z \in Z(L)$,
satisfies the equation ( \ref{ybeq2} )
$ \iff 
\theta(g,a)= \theta(a,g)=\theta(g,g)=1$, $\forall x\in L_a$ and $z\in L_g$.

The inverse operator reads: $R^{-1}(x\otimes y) = 
\alpha [y,x] \otimes z + \theta(b,a) x\otimes y $
\end{theorem}

\begin{proof} If we consider the homogeneous elements $x\in L_a$, $y\in L_b$, $t\in L_c$, 
as before,
$$ R^{12} R^{13}R^{23} (x \ot y \ot t) = R^{23} R^{13} R^{12} (x \ot y \ot t)  $$
is equivalent to
\begin{eqnarray}
\theta(a,g)[x,[y,t]]\ot z \ot z +  \theta(b,c) [[x,t],y]\ot z \ot z = \theta(g,g)  [[x,y],c]\ot z \ot z   \\
\theta(a,g)\theta(a,b+c) x\ot [y,t] \ot z = \theta(a,b)\theta(a,c)x\ot\ [y,t] \ot z    \\
\theta(b,c) \theta(a+c,b) [x,t]\ot y \ot z = \theta(a,b)\theta(b,g)[x,t]\ot y \ot z  \\   
\theta(b,c) \theta(a,c) [x,y] \ot z \ot t  = \theta(a+b,c) \theta(g,c)[x,y]\ot z \ot t 
\end{eqnarray}

Due to the conditions $\langle L_a, L_b\rangle \subseteq L_{a+b} $ the above relations   are true 
if $\theta(a,g)=\theta(b,g)=\theta(g,c)=\theta(g,g)=1$ is assumed.
\end{proof}

\section{Conclusions}

Motivated by the need to create a better frame 
for the study of Lie (super)algebras
than that presented in \cite{tbh}, this paper
extends that construction and makes an  analysis
on the constructions of solutions for the two-parameter form
of the QYBE and Yang-Baxter systems.

Following a series of posters presented at  
National Conferences on Theoretical Physics,
our paper generalizes the constructions from \cite{mj}
(to $(G,\theta)$-Lie algebras). Somewhat
less sophisticated than that of \cite{mj}
(we
do not use Category Theory), our approach
is direct and more suitable for applications. 
\cite{nichita} considered the constructions of Yang-Baxter operators
from Lie (co)algebras, suggesting  an extension
(to a bigger category with a self-dual functor acting on it) for the
duality between
the category of finite dimensional  Lie algebras
and
the category of finite dimensional  Lie coalgebras. This duality extension
was
explained in \cite{dns}.

Finally, some  applications of these results could be in constructions of: FRT
bialgebras (from the  Yang-Baxter operators obtained
in this paper), knot invariants (see Remark 2.5), solutions for the
classical Yang-Baxter equation (see below), etc. 

\begin{theorem} Let  $ ( L , [,] )$ be a Lie algebra 
and
$ z \in Z(L)$. Then:

$ r: L \ot L \ \ \longrightarrow \ \  L \ot L, \ \  
x \ot y \mapsto [x,y] \ot z $ 

satisfies the classical Yang-Baxter equation:

$ [ r^{12},\  r^{13} ] \  + \  [r^{12}, \  r^{23}] \ + \  [r^{13}, \ r^{23}] = 0 $.
\end{theorem}

\begin{center}

\end{center}
\end{document}